\documentclass[11pt,reqno]{amsart}
\usepackage{amsfonts,amsmath,amssymb,amsthm}
\newtheorem{innercustomthm}{Theorem}
\newenvironment{customthm}[1]
  {\renewcommand\theinnercustomthm{#1}\innercustomthm}
  {\endinnercustomthm}
\usepackage[margin=1.1in]{geometry}
\usepackage{enumerate}
\usepackage[dvipsnames]{xcolor} 
\usepackage[notref,notcite,color]{}
\definecolor{labelkey}{rgb}{0.5,0.5,0.5}
\usepackage{hyperref}
\hypersetup{
    colorlinks=true, 
    linktoc=all,     
    linkcolor=blue,
    citecolor=red,
}
\usepackage{soul}
\setstcolor{red}
\newtheorem{theorem}{Theorem}[section]

\newtheorem{lemma}[theorem]{Lemma}

\theoremstyle{remark}

\newtheorem{Remark}[theorem]{\bf Remark}
\newtheorem{remark}[theorem]{\bf Remark}

\numberwithin{equation}{section}

\definecolor{mygreen}{rgb}{0.0, 0.5, 0.0}
\newcommand{\blue}[1]{{\leavevmode\color{blue}#1}}

\def\de{\delta}
\def\eps{\varepsilon}
\def\la{\lambda}
\def\a{\alpha}
\def\be{\beta}

\def\ga{\gamma}
\def\part{\partial}

\newcommand{\pr}[1]{\mathbb{P}\left( #1 \right)}
\newcommand{\mean}[1]{\mathbb{E}\left[{#1}\right]}

\newcommand{\brac}[1]{\left(#1\right)}
\newcommand{\bfrac}[2]{\brac{\frac{#1}{#2}}}

\newcommand{\beq}[2]{\begin{equation}\label{#1}#2\end{equation}}
\newcommand{\ceil}[1]{\left \lceil #1 \right \rceil}

\title[Age-biased attachment graphs]{Giant descendant trees, matchings and independent sets in the age-biased attachment graphs.}
\author{Huseyin Acan}
\address{Department of Mathematics, Drexel University, Philadelphia, PA 19104}
\thanks{During the time of this research, the first coauthor was supported by NSF Fellowship (Award No.~1502650).}
\email{huseyin.acan@drexel.edu}

\author{Alan Frieze}
\address{Department of Mathematical Sciences, Carnegie Mellon University, Pittsburgh, PA 15213}
\thanks{Research of the second author supported in part by NSF Grant DMS 1661063}
\email{alan@random.math.cmu.edu}

\author{Boris Pittel}
\address{Department of Mathematics, Ohio State University, Columbus, OH 43210}
\email{bgp@math.ohio-state.edu}

\begin{document}

\begin{abstract} We study two models of an age-biased graph process: the $\delta$-version of the preferential attachment graph model (PAM) and the uniform attachment graph model (UAM), with $m$ attachments for each of incoming vertices. We show that almost surely the scaled size of a breadth-first (descendant) tree rooted at a fixed vertex converges, for $m=1$, to a limit whose distribution is a mixture of two beta-distributions and a single beta-distribution respectively, and that for $m>1$ the limit is $1$. We also analyze the likely performance of two greedy (online) algorithms, for a large matching set and a large independent set, and determine--for each model and
each greedy algorithm--both a limiting fraction of vertices involved and an almost sure convergence rate.
\end{abstract}


\maketitle
\section{Introduction}
It is widely accepted that graphs/networks are an inherent feature of life today. The classical models $G_{n,m}$ and $G_{n,p}$ of Erd\H{o}s and R\'enyi \cite{ER} and Gilbert \cite{Gi}, respectively, lacked some salient features of observed networks. In particular, they failed to have a degree distribution that decays polynomially. Barab\'asi and Albert \cite{BarAlb} suggested the Preferential Attachment Model (PAM) as a more realistic model of a ``real world'' network. There was a certain lack of rigour in \cite{BarAlb}, and later Bollob\'as, Riordan, Spencer and Tusn\'ady \cite{Bol1} gave a rigorous definition. 

Many properties of this model have been studied. Bollob\'as and Riordan \cite{BolRio} studied the diameter and proved that with high probability (whp)  PAM with $n$ vertices and $m>1$ attachments for every incoming vertex has diameter $\approx \log n/\log\log n$. Earlier result by Pittel \cite{Pit0} implied
that for $m=1$ whp the diameter of PAM is of exact order $\log n$.
Bollob\'as and Riordan \cite{BolRio3,BolRio4} studied the effect on component size from deleting random edges from PAM and showed that it is quite robust whp. The degree distribution was studied in Mori \cite{Mor0,Mor}, Flaxman, Frieze and Fenner \cite{FFF}, Berger, Borgs, Chayes and Saberi \cite{BerBorChaSab}. Pek\"oz, R\"ollin and  Ross \cite{PekRolRos} established convergence, with rate, of the joint
distribution of the degrees of finitely many vertices. Acan and Hitczenko \cite{AcaHit} found an alternative proof, without rate, via a memory game. Pittel 
\cite{Pit2} used the Bollob\'as-Riordan pairing model to approximate, with explicit error estimate, the degree sequence of the first $n^{m/(m+2)}$
vertices, $m\ge 1$, and proved that, for $m>1$, PAM is connected with probability $\approx 1-O((\log n)^{-(m-1)/3})$. Random walks on PAM have been considered in the work of Cooper and Frieze \cite{CF1,CF2}. In the first paper there are results on the proportion of vertices seen by a random walk on an evolving PAM and the second paper determines the asymptotic cover time of a fully evolved PAM. Frieze and Pegden \cite{one} used random walk in a ``local algorithm'' to find vertex 1, improving results of Borgs, Brautbar, Chayes, Khanna and Lucier \cite{BBCKL}. The mixing time of such a walk was analyzed in Mihail, Papadimitriou and Saberi \cite{Mihail} who showed rapid mixing.  Interpolating between Erd\H{o}s-R\'enyi and preferential attachment, Pittel \cite{Pit1} considered birth of a giant component
in a graph process $G_M$ on a fixed vertex set, when $G_{M+1}$ is obtained by inserting a new edge between vertices $i$ and $j$ with probability proportional to $[\text{deg}(i)+\delta]\cdot[\text{deg}(j)+\delta]$, with  $\delta>0$ being fixed. Confirming a conjecture of Pittel~\cite{Pit1}, Janson and Warnke~\cite{JW} recently determined the asymptotic size of the giant component in the supercritical phase in this graph model.

The previous paragraph gives a small sample of results on PAM that can be related to its role as a model of a real world network. It is safe to say that PAM has now been accepted into the pantheon of random graph models that can be studied purely from a combinatorial aspect. For example, Cooper, Klasing and Zito \cite{CKZ} studied the size of the smallest dominating set and Frieze, P\'erez-Gim\'enez, Pra\l{}at and Reiniger \cite{FGPR} studied the existence of perfect matchings and Hamilton cycles. 

 One source of our inspiration was the work of Mori \cite{Mor0, Mor}, see also Katona and M\'ori \cite{KatMor}, Hofstad \cite{Hof}. They were able to construct a family of martingales in a form of factorial products, with arguments being the degrees of individual vertices. This allowed them to analyze the limiting behavior of vertex degrees in a $\delta$-version of PAM. In this paper we construct a new factorial-type martingale
with an argument being the total size of a ``descendants'' subtree.  This is a generalization of the martingale
for $\delta=0$, found by Pittel \cite{Pit2}.

\section{Our Results}\label{sec:results}

For each of the two models, PAM and UAM, we study the descendants tree of a given vertex $v$; it is a {\it maximal\/} subtree rooted at $v$ and formed by increasing paths starting at $v$. 
The number of vertices in this subtree is a natural influence measure of vertex $v$.
We also analyze the performance of two on-line greedy algorithms, for finding a large matching set  and for a large independent set. 
We carry out this analysis in the context of the PAM graph process described in Bollob\'as~\cite{Bol1} and its extension taken from Hofstad \cite[Ch.\ 8]{Hof}, and the UAM graph process, see \cite{FGPR} and \cite{AcaPit}.

\subsection*{The PAM graph process, $\delta$-extension:}
Vertex $1$ has $m$ loops, so its degree is $2m$ initially. Recursively, vertex $t+1$ has $m$ edges, and it uses them one at a time either to connect to a vertex  $x\in [t]$ or to loop back on itself.

To be more precise, let us denote by $d_{t,i-1}(x)$ the degree of vertex $x$ just before the $i$-th edge of vertex $t+1$ arrives. Denoting by $w$ the random receiving end of the $i$-th edge emanating from vertex $t+1$, we have
\begin{equation}\label{defprob}
\mathbb P(w=x) =
\begin{cases}
\displaystyle \frac{d_{t, i-1}(x) +\delta}{(2m + \delta)t + 2i-1 + i\delta/m},			& 	\text{ if } x\in[t],\\
 & \\
\displaystyle \frac{d_{t,i-1}(t+1) +1+i\delta/m}{(2m + \delta)t + 2i-1 + i\delta/m},  		& 	\text{ if } x=t+1.
\end{cases}
\end{equation}

We will use the notation $\{G_{m,\delta}(t)\}$ for the resulting graph process. 

For $m=1$, writing $d_t(x)$ for $d_{t,0}(x)$, the probabilities above can be written a little more simply
\begin{equation}\label{n1}
\pr{w=x}=
\begin{cases}
\displaystyle \frac{1+\delta}{(2+\delta)t+(1+\delta)}, & \text{ if } x = t+1,\\
&\\
\displaystyle \frac{d_t(x)+\delta}{(2+\delta)t+(1+\delta)}, & \text{ if } x \in [t].
\end{cases}
\end{equation}

Bollob\'as and Riordan~\cite{BolRio} discovered the following coupling between $\{G_{m,0}(t)\}_t$ for $m>1$ and $\{G_{1,0}(mt)\}_t$. Start with the $\{G_{1,0}(t)\}$ random process and let the vertices be $v_1,v_2,\dots$. To obtain $\{G_{m,0}(t)\}$ from $\{G_{1,0}(mt)\}$, 
\begin{enumerate}
\item collapse the first $m$ vertices $v_1,\dots,v_m$ into the first vertex $w_1$ of $G_{m,0}(t)$, the next $m$ vertices $v_{m+1},\dots,v_{2m}$ into the second vertex $w_2$ of $G_{m,0}(t)$, and so on; 
\item keep the full record of the multiple edges and loops formed by collapsing the blocks $\{v_{(i-1)m+1}, \dots, v_{im}\}$ for each $i$. 
\end{enumerate}

Doing this collapsing indefinitely we get the jointly defined Bollob\'as-Riordan graph processes $\{G_{m,0}(t)\}$ and $\{G_{1,0}(mt)\}$. 
The beauty of the $\delta$-extended Bollob\'as-Riordan model is that similarly this collapsing operation applied to the process $\{G_{1,\delta/m}(mt)\}$ delivers the process $\{G_{m,\delta}(t)\}$, Hofstad~\cite{Hof}. 

(For the reader's convenience we present the explanation in Appendix.)

\begin{remark}\label{rem:delta=-m}
Note that the process is well defined for $\delta\ge -m$ since for such $\delta$, all the probabilities defined in~\eqref{defprob} are nonnegative and add up to 1. 
For $m=1$ and $\delta=-1$, it is easy to see from \eqref{n1} that there is no loop in the graph except the loop on the first vertex.
Hence, a vertex $u>1$ starts with degree 1 and then its degree does not change since as long as $d_t(u) + \delta = 0$, the vertex cannot attract any neighbors (again from \eqref{n1}).
As a result, in this case the graph is a star centered at  vertex 1.
It follows from the above coupling that $G_{m,-m}(t)$ is also a star centered at vertex $1$ and the key problems we want to solve have trivial solutions in this extreme case.
\end{remark}

\subsection*{The UAM graph process}
Conceptually close to the preferential attachment model is the uniform attachment model (UAM). In this model, vertex $t+1$ selects uniformly at random (repetitions
allowed) $m$ vertices from the set $[t]$ and attaches itself to these vertices. (See Acan and Pittel \cite{AcaPit} for connectivity and bootstrap percolation results.)  
This model can be thought of the limit of the PAM model as $\delta\to \infty$ except that loops are not allowed in this case.

\subsection{Number of Descendants}
Fix a positive integer $r$ and let $X(t)$ denote the number of descendants of $r$ at time $t$. Here $r$ is a descendant of $r$ and $x$ is a descendant of $r=O(1)$ if and only if $x$ chooses to attach itself to at least one descendant of $r$ in Step $x$. 
In other words, if we think of the graph as a directed graph with edges oriented towards the smaller vertices, vertex $x$ is a descendant of $r$ if and only if there is a directed {\it decreasing\/}  path from $x$ to $r$. In Pittel~\cite{Pit2} $X(t)$ was proposed as an influence measure of vertex $r$ at time $t$.
 
We prove two theorems.

\begin{theorem}\label{th1}
Suppose that $m=1$ and $\delta>-1$ and set $p(t):=X(t)/t$. Then almost surely (i.e. with probability $1$), $\lim p(t)$ exists, and its distribution is the mixture of two beta-distributions, with parameters
$a=1$, $b=r-\frac{1}{2+\delta}$ and $a=\frac{1+\delta}{2+\delta}$, $b=r$, weighted by $\frac{1+\delta}{(2+\delta)r-1}$ and $\frac{(2+\delta)(r-1)}{(2+\delta)r-1}$ respectively. Consequently a.s. $\liminf_{t\to\infty}p(t)>0$.
\end{theorem}

\noindent {\bf Note.\/} {\bf (i)\/} The proof is based on a new family of martingales $M_{\ell}(t):=\frac{\left(X(t)+\frac{\ga}{2+\delta}\right)^{(\ell)}}{(t+\be)^{(\ell)}}$,
$(z)^{(\ell)}$ standing for the rising factorial. This family definitely resembles the martingales Mori \cite{Mor0,Mor} used for the individual vertices' degrees.  For instance, if $D_j(t)$ is the degree of vertex $j$ at time $t\ge j$, then for some deterministic $\ga_k(t)$, $Z_{j,k}(t):= \ga_k(t)(D_j(t)+\delta)^{(k)}$ is a martingale,
Hofstad \cite{Hof}. However, $M_{\ell}(t)$ depends on $X(t)$, a {\it global\/} parameter of the PAM graph. Unsurprisingly, the proof that each $M_{\ell}(t)$ is indeed a martingale requires a very different method.
 {\bf (ii)\/} Whp $G_{1,\delta}$ is a forest of $\Theta(\log t)$ trees rooted at vertices with
loops. For the preferential attachment tree (no loops), Janson \cite{Jan} recently proved that the scaled sizes of the {\it principal\/} subtrees, those rooted at the root's children and ordered chronologically, converge a.s. to the {\it GEM\/} distributed random variables. His techniques differ significantly.

For $m>1$, we use Theorem \ref{th1} to prove a somewhat surprising result that, for $r=O(1)$, almost surely all but a vanishingly small fraction of vertices are descendants of the vertex $r$, (cf.~\cite{Jan}).
\begin{theorem}\label{th2} 
Let $m>1$ and $\delta>-m$ and let $p_X(t)=X(t)/t$, $p_Y(t)=Y(t)/(2mt)$, where $Y(t)$ is the total degree of the descendants of $r$ at time $t$.
Then almost surely  $\lim_{t\to\infty} p_X(t)=\lim_{t\to\infty} p_Y(t)=1$. 
\end{theorem}


For the case of UAM, we have the following result.

\begin{theorem}\label{th2+}
Consider the UAM graph process $G_{t,m}$. Given \blue{$r>1$}, let $X(t)$ be the cardinality of the descendant tree rooted at vertex \blue{$r$}, and let $p(t):=X(t)/t$. 
\begin{enumerate}
\item[\textup{(i)}] For $m=1$, almost surely, $\lim p(t)$ exists and
it has the same distribution as the minimum of $(r-1)$ independent $[0,1]$-Uniforms. Consequently a.s. $\liminf_{t\to\infty}p(t)>0$.
\item[\textup{(ii)}] For $m>1$, almost surely $\lim_{t\to\infty} p(t)=1$.
\end{enumerate}
\end{theorem}

\subsection{Greedy Matching Algorithm}
We analyze a greedy matching algorithm;  a.s. it delivers a surprisingly large matching set even for relatively small $m$.  This algorithm generates the increasing sequence $\{ M(t)\}$ of partial matchings on the sets $[t]$, with $ M(1)=\emptyset$. Suppose that $X(t)$ is the set of unmatched vertices in $[t]$ at time $t$. If $t+1$ attaches itself to a vertex $u\in X(t)$, then $M(t+1)=M(t)\cup\{\{u,t+1\}\}$, otherwise $M(t+1)=M(t)$. (If $t+1$ chooses multiple vertices from $X(t)$, then we pick one of those as $u$ arbitrarily.) 

Consider first the PAM graph.
Let 
\[
h(z)=h_{m,\de}(z):= 2\left[1-\bfrac{m+\de}{2m+\de}z\right]^m-z-1,
\]
and let $\rho=\rho_{m,\de}$ be  the unique root $\rho=\rho_{m,\de}$ in the interval $[0,1]$ of $h(z)=0$: $\rho_{m,\delta}\in (0,1)$ if $\delta>-m$.  
\begin{theorem}\label{th3} 
Let $M(t)$ and $X(t)$ be the set of greedy matching and the set of uncovered vertices at time $t$, and let $x(t)=X(t)/t$.
For any $\delta>-m$ and $\a<1/3$, almost surely,
\[
\lim_{t\to\infty} t^{\a}\max\{0, x(t)-\rho_{m,\delta}\}=0.
\] 
In consequence, the Greedy Matching Algorithm a.s. finds a sequence of nested matchings $\{M(t)\}$, with $M(t)$ of size $(1-o(1))(1-\rho_{m,\delta})t/2$, at least. 
\end{theorem}

\begin{Remark}
Observe that $\rho_{m,-m}=1$, which makes it plausible that the maximum matching size is minuscule compared to $t$. 
In fact, by Remark~\ref{rem:delta=-m}, $G_{m,-m}(t)$ is the star centered at vertex $1$ and hence the maximum matching size is 1.
\end{Remark}

\begin{Remark} Consider the case $\de=0$. Let $r_m:=1-\rho_{m,0}$; some values of $r_m$ are:
\begin{equation}\label{rvalues}
 \begin{alignedat}{3}
&r_1=0.5000, \qquad &&r_2=0.6458, \qquad &&r_5=0.8044,\\
&r_{10}=0.8863,&&r_{20}=0.9377, &&r_{70}=0.9803.
\end{alignedat}
\end{equation}
With a bit of calculus, we obtain that $r_m=1- 2m^{-1}\log2 +O(m^{-2})$.
\end{Remark}

\begin{theorem}\label{thm4} Let $M(t)$ denote the greedy matching set after $t$ steps of the UAM process. 
Let $r_m$ denote a unique positive root of $2(1-z^m) -z=0$: $r_m=1-m^{-1}\log 2+O(m^{-2})$. 
Then, for any $\a<1/3$, almost surely
\[
\lim_{t\to\infty} t^{\a}\biggl| \frac{2|M(t)|}{t}-r_m\biggr|=0.
\]
\end{theorem} 

\noindent Some values of $r_m$ in this case are:
\begin{alignat*}{3}
&r_1=0.6667, \qquad &&r_2=0.7808, \qquad &&r_5=0.8891,\\
&r_{10}=0.9386, &&r_{20}=0.9674, &&r_{35}=0.9809.
\end{alignat*}

\subsection{Greedy Independent Set Algorithm} The algorithm generates an increasing sequence of independent sets 
$\{I(t)\}$ on vertex sets $[t]$. Namely, $I(1)=\{1\}$, and $I(t+1)=I(t)\cup\{t+1\}$ if $t+1$ does not select any of the vertices in $I(t)$; if it does, then $I(t+1)=I(t)$. $I(t)$ is also a dominating set for the PAM/UAM graph with vertex set $[t]$;
indeed if a vertex $\tau\in [t]\setminus I(t)$ did not have any neighbor in $I(t)$, then vertex $\tau$ would have been added to $I(\tau-1)$ at step $\tau$. (Pittel \cite{Pit-1} analyzed performance of this algorithm applied to Erd\H os--R\'enyi random
graph with a large, but fixed vertex set.)

For the PAM case we prove
\begin{theorem}\label{thm5} Let $w_m$ denote the unique root of $-w+(1-w)^m$ in $(0,1)$. For any 
$\chi\in \Bigl(0,\min\Bigl\{\frac{1}{3},\frac{2m+2\delta}{3(2m+\delta)}\Bigr\}\Bigr)$, almost surely
\begin{equation}\label{PAM0}
\lim_{t\to\infty}t^{\chi}\biggl|\frac{|I(t)|}{t}-w_m\biggr|=0.
\end{equation}
\end{theorem}

\begin{remark} Thus the limiting scaled size of the greedy independent set does not depend on $\delta$, but
the convergence rate does.
\end{remark}

For the UAM case we prove an almost identical
\begin{theorem}\label{thm6} Let $w_m$ be the unique positive root of $-w+(1-w)^m$ in $(0,1)$. Then, for any $\a<1/3$, almost surely
\[
\lim_{t\to\infty}t^{\a}\Big|\frac{|I(t)|}{t} - w_m\Big|=0.
\]
\end{theorem}

\begin{remark}\label{w_m}

Let $w_m$ the unique positive root of $(1-w)^m-w$ in $(0,1)$. As $m\to\infty$,
\[
w_m= \frac{\log m}{m}+O\bigl(m^{-1}\log\log m\bigr).
\]

So, for $m$ large and both models, a.s. for all large enough $t$ the greedy algorithm delivers an independent set containing a fraction $\sim\frac{\log m}{m}$
of all vertices in $[t]$. It was proved in \cite{FGPR} that for each large $t$ with probability $1-o(1))$ the fraction of vertices in the largest independent set in the PAM graph process and in the UAM graph process is at most $(4+o(1))\frac{\log m}{m}$ and $(2+o(1))\frac{\log m}{m}$,
respectively. We conjecture that for each of the processes there exists a corresponding constant $c$ such that, for
$m\to\infty$, a.s. for all large $t$ the largest independent set contains a fraction $\sim c\frac{\log m}{m}$ of all $t$ vertices.

Since $I(t)$ is dominating, our results prove, for $m$ large, a.s. existence for all large $t$ of relatively small
dominating sets, of cardinality $\sim t\frac{\log m}{m}$.
\end{remark}

\section{Descendant trees}\label{sec:descendant}
Instead of referring the reader back to ``Our results'' section, we start this, and other proof sections,
with formulating in full the claim in question.

\subsection*{Proof of Theorem~\ref{th1}}

\begin{customthm}{2.2}
Suppose that $m=1$ and $\delta>-1$ and set $p(t):=X(t)/t$. Then almost surely (i.e. with probability $1$), $\lim p(t)$ exists, and its distribution is the mixture of two beta-distributions, with parameters
$a=1$, $b=r-\frac{1}{2+\delta}$ and $a=\frac{1+\delta}{2+\delta}$, $b=r$, weighted by $\frac{1+\delta}{(2+\delta)r-1}$ and $\frac{(2+\delta)(r-1)}{(2+\delta)r-1}$ respectively. Consequently a.s. $\liminf_{t\to\infty}p(t)>0$.
\end{customthm}

\begin{proof}For $t\ge r>1$, let $X(t)=X_{m,\delta}(t)=X_{m,\delta}(t,r)$ and $Y(t)=Y_{m,\delta}(t)=Y_{m,\delta}(t,r)$ denote the size and the total degree of the vertices in the vertex set of the subtree $T(t)=T_{m,\delta}(t,r)$ rooted at $r$; so $X(r)=1$ and $Y(r)\in [m,2m]$, where $m$ ($2m$ resp.) is attained when vertex $r$ forms no loops (forms $m$ loops resp.) at itself. Introduce  $p(t)=p_Y(t)=\frac{Y(t)}{2mt}$ and $p_X(t)=\frac{X(t)}{t}$. This notation will be used in the proof of Theorem \ref{th2} as well, but of course $m=1$ Theorem~\ref{th1}.

Here
\[
Y(t)=\left\{\begin{aligned}
&2X(t),&&\text{if }r\text{ looped on itself},\\
&2X(t)-1,&&\text{if }r\text{ selected a vertex in }[r-1].\end{aligned}\right.
\]
(In particular, $p_X(t) =p(t)+O(t^{-1})$.) So, by \eqref{n1},
\begin{equation*}
\begin{aligned}
\mathbb P(X(t+1)&=X(t)+1|\circ)=\frac{Y(t)+\delta X(t)}{(2+\delta)t +(1+\delta)}\\
&=\left\{\begin{aligned}
&\frac{(2+\delta)X(t)}{(2+\delta)t+(1+\delta)},&&\text{if }r\text{ looped on itself},\\
&\frac{(2+\delta)X(t)-1}{(2+\delta)t+(1+\delta)},&&\text{if }r\text{ selected a vertex in }[r-1].\end{aligned}\right.
\end{aligned}
\end{equation*}
Thus we are led to consider the process $X(t)$ such that
\begin{align*}
\mathbb P(X(t+1)=X(t)+1|\circ)   &=   \frac{(2+\delta)X(t)+\ga}{(2+\delta)t+(1+\delta)}\\
\mathbb P(X(t+1)=X(t)|\circ)   &=   1 - \mathbb P(X(t+1)=X(t)+1|\circ) ,
\end{align*}
$\ga=0$ if $r$ looped on itself, $\ga=-1$ if $r$ selected a vertex in $[r-1]$.
Letting $\be=\frac{1+\delta}{2+\delta}$, the above equation can be written as
\beq{jumpup}{
\begin{aligned}
\mathbb P(X(t+1)=X(t)+1|\circ)   &=   \frac{X(t) + \ga/(2+\delta)}{t+\be}\\
\mathbb P(X(t+1)=X(t)|\circ)   &=   1 - \frac{X(t) + \ga/(2+\delta)}{t+\be}.
\end{aligned}
}

For $\delta=0$ the following claim was proved in Pittel \cite{Pit2}. We denote by $z^{(\ell)}$ the rising factorial $\prod_{j=0}^{\ell-1}(z+j)$.

\begin{lemma}\label{Lem1} 
Let $\be=\frac{1+\delta}{2+\delta}$ and $Z(t) = X(t) + \frac{\ga}{2+\delta}$. Then, conditioned on the attachment record during the time interval
$[r,t]$, i.e. starting with attachment decision by vertex $r$, we have
\[
\mean{Z(t+1)^{(\ell)}\big| \circ}=  \bfrac{t+\be+\ell}{t+\be} Z(t)^{(\ell)}.
\]
Consequently $M(t):=\frac{Z(t)^{(\ell)}}{(t+\be)^{(\ell)}}$ is a martingale.
\end{lemma}

\begin{proof} 
By \eqref{jumpup}, we have: for $k\ge 1$, and $t\ge r$,
\begin{equation}\label{2n}
\begin{aligned}
\mathbb E[Z^k(t+1)|\circ]&=(Z(t)+1)^k\,\frac{Z(t)}{t+\be}+Z^k(t)\left(1-\frac{Z(t)}{t+\be}\right)\\
&=\frac{Z(t)}{t+\be}\sum_{j=0}^k\binom{k}{j} Z^j(t)+Z^k(t)\left(1-\frac{Z(t)}{t+\be}\right)\\
&=Z^k(t)+\frac{Z(t)}{t+\be}\sum_{j=0}^{k-1}\binom{k}{j} Z^j(t)\\
&=Z^k(t)\frac{t+\be+k}{t+\be}+\frac{1}{t+\be}\sum_{j=1}^{k-1}\binom{k}{j-1}Z^j(t).				
\end{aligned}
\end{equation}
Next recall that
\begin{equation}\label{Com1}
z^{(\ell)}=\sum_{k=1}^{\ell} z^k s(\ell,k),
\end{equation}
where $s(\ell,k)$ is the signless, first-kind, Stirling number, i.e. the number of permutations of the set $[\ell]$ with $k$ cycles. In particular,
\begin{equation}\label{Com}
\sum_{\ell\ge 1}\eta^{\ell}\frac{s(\ell,k)}{\ell!}=\frac{1}{k!}\log^k\frac{1}{1-\eta},\quad |\eta|<1,
\end{equation}
Comtet~\cite[Section 5.5]{Com}. Using \eqref{2n}  and \eqref{Com1}, we have
\begin{align*}
&\mathbb E\bigl[Z^{(\ell)}(t+1)|\circ\bigr]= \sum_{k=1}^{\ell} s(\ell,k) \mathbb E\bigl[Z^k(t+1)|\circ\bigr]\\
&= (t+\be)^{-1}\sum_{k=1}^{\ell}s(\ell,k)\cdot
\Biggl(\!(t+\be+k) Z^k(t)+\sum_{j=0}^{k-1}\binom{k}{j-1} Z^j(t)\Biggr)			 \\
&=:(t+\be)^{-1}\sum_{i=1}^{\ell} \sigma(\ell,i)Z^i(t),\\
\sigma(\ell,i)&\!=\!\left\{\begin{aligned}
&(t+\be+\ell) s(\ell,\ell),&&\text{if }i=\ell,\\
&(t+\be) s(\ell,i)+\sum_{k=i}^{\ell} s(\ell,k)\binom{k}{i-1},&&\text{if }i<\ell.\end{aligned}\right.
\end{align*}
We need to show that $\sigma(\ell,i)=(t+\be+\ell)s(\ell,i)$ for $k<\ell$, which is equivalent to
\[
\ell s(\ell,i)=\sum_{k=i}^{\ell}s(\ell,k)\binom{k}{i-1}.
\]
To prove the latter identity, it suffices to show that, for a fixed $i$, the exponential generating functions of the two sides coincide. By \eqref{Com},
\begin{align*}
&\quad\sum_{\ell\ge 1}\frac{\eta^{\ell}}{\ell!}\sum_{k=i}^{\ell}s(\ell,k)\binom{k}{i-1}=\sum_{k\ge i}\binom{k}{i-1}\sum_{\ell\ge k}\frac{\eta^{\ell}}{\ell!} s(\ell,k)
\\
&=\sum_{k\ge i}\binom{k}{i-1} \frac{1}{k!}\log^k\frac{1}{1-\eta}=\frac{1}{(i-1)!}\left(\log^{ -1}\frac{1}{1-\eta}\right) \sum_{s\ge 1}\frac{1}{s!}\log^s\frac{1}{1-\eta}\\
&=\frac{1}{(i-1)!}\left(\log^{i-1}\frac{1}{1-\eta}\right)\left(\frac{1}{1-\eta}-1\right)=\frac{1}{(i-1)!}\left(\log^{i-1}\frac{1}{1-\eta}\right)\frac{\eta}{1-\eta}.
\end{align*}
And, using \eqref{Com} again,
\begin{align*}
&\sum_{\ell\ge 1}\frac{\eta^{\ell}}{\ell!}\,\ell s(\ell,i)=\eta\sum_{\ell\ge 1}\frac{\ell\eta^{\ell-1}}{\ell!}\,s(\ell,i)\\
&=\eta\frac{d}{d\eta}\Biggl(\frac{1}{i!}\log^i\frac{1}{1-\eta}\Biggr)=\frac{1}{(i-1)!}\left(\log^{i-1}\frac{1}{1-\eta}\right)\frac{\eta}{1-\eta}. \qedhere
\end{align*}
\end{proof}
To identify the $\lim_{t\to\infty} p(t)$, recall that the classic beta probability distribution has density 
\[
f(x;a,b)=\frac{\Gamma(a+b)}{\Gamma(a)\Gamma(b)}x^{a-1}(1-x)^{b-1},\quad x\in (0,1),
\]
parametrized by two parameters $a>0$, $b>0$, and moments
\begin{equation}\label{moments}
\int_0^1 x^{\ell} f(x;a, b)\,dx=\prod_{j=0}^{\ell-1}\frac{a+j}{a+b+j}.
\end{equation}
We can now complete the proof of Theorem \ref{th1}. 
By Lemma \ref{Lem1}, we have $\mathbb E [M(t)|\ga]=M(r)$, i.e.
\[
\mathbb E\Biggl[\frac{\left(X(t)+\frac{\ga}{2+\delta}\right)^{(\ell)}}{(t+\be)^{(\ell)}}\,\Bigg|\ga\Biggr]=\frac{\left(1+\frac{\ga}{2+\delta}\right)^{(\ell)}}{(r+\be)^{(\ell)}}.
\]
For every $\ell\ge 1$, by martingale convergence theorem, conditioned on $\ga$, there exists an integrably  finite $\mathcal M_{\ga,\ell}$ such that a.s.
\[
\lim_{t\to\infty}\frac{\left(X(t)+\frac{\ga}{2+\delta}\right)^{(\ell)}}{(t+\be)^{(\ell)}}=\mathcal M_{\ga,\ell},\quad \ell\ge 0,
\]
and
\[
\mathbb E[\mathcal M_{\ga,\ell}]=\frac{\left(1+\frac{\ga}{2+\delta}\right)^{(\ell)}}{(r+\be)^{(\ell)}}.
\]
So, using the notation $p_X(t)=X(t)/t$, we have: a.s.
\begin{equation}\label{m=1}
\begin{aligned}
\lim_{t\to\infty}(p_X(t))^{\ell}=\mathcal M_{\ga,\ell}=(\mathcal M_{\ga,1})^{\ell},
\end{aligned}
\end{equation}
and
\[
\mathbb E\bigl[(\mathcal M_{\ga,1})^{\ell}\bigr]=\frac{\left(1+\frac{\ga}{2+\delta}\right)^{(\ell)}}{(r+\be)^{(\ell)}}=\prod_{j=0}^{\ell-1}\frac{1+\frac{\ga}{2+\delta}+j}{r+\be+j}.
\]
This means that $\mathcal M_{\ga,1}$ is beta-distributed with parameters $1+\frac{\ga}{2+\delta}$ and $r+\be-1-\frac{\ga}{2+\delta}$.  By the definition of
$\ga$ and \eqref{n1}, we have
\[
\mathbb P(\ga=0)=\frac{1+\delta}{(2+\delta)(r-1)+(1+\delta)}=\frac{1+\delta}{2r-1+\delta r}.
\]
We conclude that $\lim_{t\to\infty}p(t)$ has the distribution which is the mixture of the two beta distributions, with parameters $a=1$, $b=r-\frac{1}{2+\delta}$, and $a=\frac{1+\delta}{2+\delta}$, $b=r$, weighted by $\frac{1+\delta}{(2+\delta)r-1}$ and 
$\frac{(2+\delta)(r-1)}{(2+\delta)r-1}$ respectively.
This completes the proof of Theorem~\ref{th1}.
\end{proof}

\subsection{Proof of Theorem \ref{th2}}

\begin{customthm}{2.3}
Let $m>1$ and $\delta>-m$ and let $p_X(t)=X(t)/t$, $p_Y(t)=Y(t)/(2mt)$, where $Y(t)$ is the total degree of the descendants of $r$ at time $t$.
Then almost surely  $\lim_{t\to\infty} p_X(t)=\lim_{t\to\infty} p_Y(t)=1$. 
\end{customthm}

\begin{proof}We need to derive tractable formulas/bounds for the conditional distribution of $Y(t+1)-Y(t)$. First, let us evaluate the 
conditional probability that selecting the second endpoints of the $m$ edges incident to vertex  $t+1$ no loops will be formed.
Suppose there has been no loop in the first $i-1$ steps, $i\in [m]$; call this event $\mathcal E_{i-1}$.  On event $\mathcal E_{i-1}$, as the $i$-th edge incident to $t+1$ is about to attach its second end to a vertex in $[t]\cup\{t+1\}$, the total degree of
all these vertices is $2mt+i-1$ ($1\le i\le m$). So, by the definition of the transition probabilities (items {\bf (a), (b), (c)\/}) we have
\[
\mathbb P(\mathcal E_i |\circ)= \frac{2mt +i-1+t\delta}{2mt+2(i-1)+t\delta +1+\frac{i\delta}{m}},
\]
``$\circ$'' indicating conditioning on the full record of $i-1$ preceding attachments such that the event $\mathcal E_{i-1}$ holds. Crucially this conditional probability
depends on $i$ only. Therefore the probability of a given full {\it loops-free\/} record of the $m$ attachments is equal to the corresponding probability for the ``no loops in $m$ attachments process'',  multiplied by
\begin{equation}\label{Pmt=}
 \Pi_m(t):=\prod_{i=1}^m\frac{2mt +i-1+t\delta}{2mt+2(i-1)+t\delta +1+\frac{i\delta}{m}}=1-O(t^{-1}).
\end{equation}

\begin{lemma}\label{lem1}  If no loops are allowed in the transition from $t$ to $t+1$, then 
for $a\in [m]$,
\[
\pr{Y(t+1) =Y(t)+m+a \mid \circ}
=\binom{m}{a}   \frac{(Y(t)+\delta X(t))^{(a)} \cdot \bigl(2mt-Y(t)+\delta(t-X(t))^{(m-a)}}  {\bigl((2m+\delta)t\bigr)^{(m)}}
\]
and 
\[
\pr{Y(t+1) =Y(t)\mid \circ} = \frac{ \bigl(2mt-Y(t)+\delta(t-X(t))^{(m)}}  {\bigl((2m+\delta)t\bigr)^{(m)}}.
\]
\end{lemma}

\begin{proof} Vertex $t+1$ selects, in $m$ steps, a sequence $\{v_1,\dots, v_m\}$ of $m$ vertices from $[t]$, with $t$ choices for every selection. Introduce $\Bbb{\boldsymbol{I}}=\{\Bbb I_1,\dots,\Bbb I_m\}$, where $\Bbb I_i$ is the indicator of the event $\{v_i\in V(T(t))\}$. The total vertex degree
of $[t]$ (of $V(T(t))$ respectively) right before step $i$ is $2mt + i-1$ $\bigl(Y(t)+\mu_i\text{ respectively}, \mu_i:=|\{j<i: \Bbb I_j=1\}|\bigr)$. Conditioned on this
prehistory, 
\begin{align*}
\mathbb P(\Bbb I_i=1)&=\frac{Y(t)+\delta X(t)+\mu_i}{2mt+\delta t+i-1},\\
 \mathbb P(\Bbb I_i=0)&=\frac{2mt-Y(t)+\delta(t-X(t))+i-1-\mu_i}{2mt+\delta t+ i -1}.
\end{align*}
Therefore a sequence $\Bbb{\boldsymbol{I}}$ will be the outcome of the $m$-step selection with probability
\begin{multline*}
\qquad\qquad\qquad\mathbb P(\Bbb{\boldsymbol{I}})=\left(\prod_{i\in [m]} \Bigl((2m+\delta)t +i-1\bigr)\Bigr)\right)^{-1}\\
\times\prod_{i:\, \Bbb I_i=1}\Bigl(Y(t)+\delta X(t)+\mu_i\Bigr)\,\,\cdot \prod_{i:\, \Bbb I_i=0}\Bigl(2mt-Y(t)+\delta(t-X(t))+i-1-\mu_i\Bigr).
\end{multline*}
 Furthermore, for $a\in [m]$, on the event $\{Y(t+1)=Y(t)+m+a\}$ for each admissible $\Bbb{\boldsymbol{I}}$ we have 
 \[
 \{\mu_i\}=\{0,1,\dots, a-1\},\quad \{i-1-\mu_i\}=\{0,1,\dots, m-a-1\},
\]
so that
\[
\mathbb P(\Bbb{\boldsymbol{I}})= \frac{(Y(t)+\delta X(t))^{(a)} \bigl(2mt-Y(t)+\delta(t-X(t))^{(m-a)}}{\bigl((2m+\delta)t\bigr)^{(m)}}.
\]
Since the total number of admissible sequences $\Bbb{\boldsymbol{I}}$ is $\binom{m}{a}$, we obtain the first formula in Lemma~\ref{lem1}. The second formula is the case of
$\mathbb P(\Bbb{\boldsymbol{I}})$ with ${a=0}. \qedhere$ 
\end{proof} 

It is clear from the proof of Lemma \ref{lem1} that $\{\mathbb P_m(a)\}_{0\le a\le m}$,
\[
\mathbb P_m(a):=\binom{m}{a}\frac{(Y(t)+\delta X(t))^{(a)} \bigl(2mt-Y(t)+\delta(t-X(t))^{(m-a)}}{\bigl((2m+\delta)t\bigr)^{(m)}},
\]
is a probability distribution of a random variable $D$, a ``rising-factorial'' counterpart of the binomial  $\mathcal D=\text{Bin}(m,p=Y(t)/2mt)$. 
Define the falling factorial $(x)_{\ell}=x(x-1)\cdots (x-\ell+1)$. It is well known that $\mathbb E[(\mathcal D)_{\mu}]=(m)_{\mu} p^{\mu}$, $(\mu\le m)$. 
For $D$ we have
\begin{multline}\label{2.2}
\mathbb E[(D)_{\mu}]=\sum_a(a)_{\mu} \mathbb P_m(a)=\frac{(m)_{\mu}\,(Y(t)+\delta X(t))^{(\mu)}}{\bigl((2m+\delta)t\bigr)^{(\mu)}}\cdot \sum_{a\ge \mu}\binom{m-\mu}{a-\mu}\\
\times \frac{(Y(t)+\delta X(t)+\mu)^{(a-\mu)}\bigl((2m+\delta)t+\mu-(Y(t)+\delta X(t)+\mu))^{((m-\mu)-(a-\mu))}}{(2mt+\mu)^{(m-\mu)}}\\
=\frac{(m)_{\mu}\,(Y(t)+\delta X(t))^{(\mu)}}{\bigl((2m+\delta)t\bigr)^{(\mu)}},
\end{multline}
since the sum over $a\ge \mu$ is $\sum_{\nu\ge 0}\mathbb P_{m-\mu}(\nu)=1$.

From Lemma \ref{lem1} and \eqref{2.2} we have: if no loops during the transition from $t$ to  $t+1$ are allowed, then
\begin{align}\label{2.3-}
\mathbb E[Y(t+1)-Y(t)|\circ] &=\sum_{a=1}^m (a+m) \mathbb P_m(a) \notag\\
&=\frac{m\bigl(Y(t)+\delta X(t)\bigr)}{(2m+\delta)t}+ m\left(1-\frac{\bigl((2m+\delta)t-Y(t)-\delta X(t)\bigr)^{(m)}}{\bigl((2m+\delta)t\bigr)^{(m)}}\right),
\end{align}
and
\begin{equation}\label{2add}
\mathbb E[X(t+1)-X(t)|\circ]=1-\frac{\bigl((2m+\delta)t-Y(t)-\delta X(t)\bigr)^{(m)}}{\bigl((2m+\delta)t\bigr)^{(m)}}.
\end{equation}
What if the ban on loops at the vertex $t+1$ is lifted? From the discussion right before Lemma \ref{lem1}, we see that 
both 
$\mathbb E[\bigl(Y(t+1)-Y(t)\bigr) \mathbb I(\text{no loops})|\circ]$ and 
$\mathbb E[\bigl(X(t+1)-X(t)\bigr) \mathbb I(\text{no loops})|\circ]$ 
are equal
to the respective RHS's in \eqref{2.3-} and \eqref{2add} {\it times\/} $\Pi_m(t)=1-O(t^{-1})$. Consequently, adding the terms $O(t^{-1})$ to
the RHS of \eqref{2.3-} and to the RHS of \eqref{2add} we obtain the sharp asymptotic formulas for $\mathbb E[Y(t+1)-Y(t)|\circ]$ and 
$\mathbb E[X(t+1)-X(t)|\circ]$ in the case of the loops-allowed model. We will also need
\[
p(t)=\frac{2m}{2m+\delta}\,p_Y(t)+\frac{\delta}{2m+\delta}\,p_X(t),
\]
where $p_Y(t)=\frac{Y(t)}{2mt}$ and $p_X(t)=\frac{X(t)}{t}$ as defined in the beginning of the section.

To continue the proof of Theorem~\ref{th2},   
we note first that $mX(t)\le Y(t)\le 2mX(t)$. The lower bound is obvious. The upper bound follows from induction on $t$: Suppose $Y(t)\le 2mX(t)$. If $X(t+1)=X(t)$, then $Y(t+1)=Y(t)\le 2mX(t)=2m X(t+1)$. If $X(t+1)=X(t)+1$, then $Y(t+1)\le Y(t)+2m\le 2mX(t)+2m=2mX(t+1)$.
Therefore, by the definition of $p(t)$,  we have
\begin{equation}\label{double}
\frac{p_X(t)}{2}\le p_Y(t)\le p_X(t) \Longrightarrow \frac{m+\delta}{2m+\delta}\,p_X(t)\le p(t) \le p_X(t);
\end{equation}
in particular, $p(t)\in [0,1]$ since $\delta\ge -m$.
We will also need
\[
\frac{\bigl((2m+\delta)t-Y(t)-\delta X(t)\bigr)^{(m)}}{\bigl((2m+\delta)t\bigr)^{(m)}} = (1-p(t))^m +O(t^{-1}).
\]
So, using \eqref{2.3-}, we compute
\begin{multline}\label{pY(t+1)}
\mathbb E[p_Y(t+1)|\circ]=\mathbb E\left[\frac{Y(t+1)}{2mt}\cdot \frac{t}{t+1}\Big|\circ\right]\\
=\frac{t}{t+1}\!\left(\!p_Y(t)+\frac{1}{2t}\bigl[1+p(t)-(1-p(t))^m\bigr]+O(t^{-2})\!\!\right)=p_Y(t)+q_Y(t),\\
 q_Y(t):=\frac{1}{2(t+1)}\bigl[1+p(t)-2p_Y(t)-(1-p(t))^m\bigr]+O(t^{-2}).
\end{multline}
Likewise
\begin{equation}\label{pX(t+1)}
\begin{aligned}
&\qquad\qquad\mathbb E[p_X(t+1)|\circ]=p_X(t)+q_X(t),\\
&q_X(t)=\frac{1}{t+1}\bigl[1-p_X(t)-(1-p(t))^m\bigr]+O(t^{-2}).
\end{aligned}
\end{equation}
Multiplying the equation \eqref{pY(t+1)} by $\frac{2m}{2m+\delta}$, the equation \eqref{pX(t+1)} by $\frac{\delta}{2m+\delta}$, and adding them, we obtain
\begin{equation}\label{p(t+1)}
\begin{aligned}
&\qquad\qquad\qquad\quad\mathbb E[p(t+1)|\circ]=p(t) +q(t),\\
&q(t):=\frac{m+\delta}{(2m+\delta)(t+1)}\bigl[1-p(t)-(1-p(t))^m\bigr]+O(t^{-2}).
\end{aligned}
\end{equation}
From the first line in \eqref{p(t+1)} it follows that  
\[
\sum_{t=1}^{\tau}\Bbb E[p(t+1)]=\sum_{t=1}^{\tau}\Bbb E[p(t)]+\sum_{t=1}^{\tau}\Bbb E[q(t)],
\]
implying that
\[
\limsup_{\tau\to\infty}\sum_{t\le\tau} \mathbf E[q(t)]\le \limsup_{\tau\to\infty}\Bbb E[p(\tau+1)]\le 1.
\]
Since $1-z-(1-z)^m\ge 0$ on $[0,1]$, the second line in \eqref{p(t+1)} implies that $|q(t)|\le q(t)+O(t^{-2})$.
Since $\sum_t t^{-2}<\infty$, we see that $\sum_t\mathbb E[|q(t)|]<\infty$. 

So a.s. there exists $Q:=\lim_{\tau\to\infty}\sum_{1\le t\le \tau} q(t)$, with $\mathbb E[|Q|]
\le \sum_t\mathbb E[|q(t)|]<\infty$, i.e. a.s. $|Q|<\infty$.  Introducing $Q(t+1)=\sum_{\tau\le t}q(\tau) $, we see from \eqref{p(t+1)} that $\{p(t+1)-Q(t+1)\}_{t\ge 1}$ is a martingale with $\sup_t |p(t+1)-Q(t+1)|\le 1 + \sum_{\tau\ge 1} |q(\tau)|$. By the martingale convergence theorem we obtain that there exists an integrable $\lim_{t\to\infty}(p(t)-Q(t))$, implying that a.s. there exists a random $p(\infty)=\lim_{t\to\infty}p(t)$. The \eqref{p(t+1)} also implies that
\[
1\ge \mathbb E[p(\infty)]=\frac{m+\delta}{2m+\delta}\sum_{t\ge 1}\frac{1}{t+1}\mathbb E\bigl[1-p(t)-(1-p(t))^m\bigr]+ O(1).
\]
Since $m+\delta>0$ and 
\[
\lim_{t\to\infty} \mathbb E\bigl[1-p(t)-(1-p(t))^m\bigr] =\mathbb E\bigl[1-p(\infty)-(1-p(\infty))^m\bigr],
\]
and the series $\sum_{t\ge 1} t^{-1}$ diverges, we obtain that $\mathbb P(p(\infty)\in \{0,1\})=1$. 

Recall  that  $p(t)\ge \frac{m+\delta}{2m+\delta}\, p_X(t)$. If we show that a.s. $\liminf_{t\to\infty}p_X(t)>0$, it will follow that
a.s.  $p(\infty)>0$, whence a.s. $p(\infty)=1$, implying (by $p(t)\le p_X(t)$) that a.s. $p_X(\infty)$ exists, and is $1$, and consequently
(by the formula for $p(t)$) a.s. $p_Y(\infty)$ exists, and is $1$.

So let's prove that a.s. $\liminf_{t\to\infty}p_X(t)>0$. 
Recall that we did prove the latter for $m=1$. To transfer this earlier result to $m>1$, we need to establish 
some kind of monotonicity with respect to $m$. The coupling described in Section~\ref{sec:results} to the rescue!

\begin{lemma}\label{lem:coupling}
For the coupled processes $\{G_{m,\delta}(t)\}$ and $\{G_{1,\delta/m}(mt)\}$, we have $X_{m,\delta}(t,r) \ge m^{-1} X_{1,\delta/m}(mt,mr)$.
\end{lemma}

\begin{proof}
Let us simply write $G_1$ and $G_m$ for the two graphs $G_{1,\delta/m}(mt)$ and $G_{m,\delta}(t)$, respectively. 
Similarly, write $T_1$ and $T_m$, respectively, for the descendant tree in $G_{1,\delta/m}(mt)$ rooted at $mr$ and the descendant tree in $G_{m,\delta}(t)$ rooted at $r$. 
If $v_a\in T_1$, i.e. $v_a$ is a descendant of $mr$, then for $b=\ceil{a/m}$ we have $w_b=\{v_{m(b-1)+i}\}_{i\in [m]}\ni v_a$, implying that
$w_b$ is a descendant of $r$ in $G_m$, i.e. $w_b\in T_m$. (The converse is generally false: if $w_b$ is a descendant of $r$, it does not mean that
every $v_{m(b-1)+i}$, ($i\in [m]$), is a descendant of $mr$.) Therefore
\[
X_{m,\delta}(t,r)=|V(T_m)|\ge m^{-1} |V(T_1)|=m^{-1}X_{1,\delta/m}(mt,mr). \qedhere
\]
\end{proof}
Thus, to complete the proof of the theorem,  i.e. for $\delta>-m$, we {\bf (a)\/} use Theorem \ref{th1}, to assert that for the process $\{G_{1,\delta/m}(t)\}$, a.s. $\lim_{t\to\infty} p_X (t)>0$; {\bf (b)\/} use Lemma \ref{lem:coupling}, to assert that a.s. $\liminf_{t\to\infty} p_X (t)>0$ for $\{G_{m,\delta}(t)\}$ as well. The proof of Theorem \ref{th2} is complete.

\end{proof}

\subsection{Proof of Theorem \ref{th2+}} 
\begin{customthm}{2.4}
Consider the UAM graph process $G_{t,m}$. Given $r>1$, let $X(t)$ be the cardinality of the descendant tree rooted at vertex $r$, and let $p(t):=X(t)/t$. 
\begin{enumerate}
\item[\textup{(i)}] For $m=1$, almost surely, $\lim p(t)$ exists and
it has the same distribution as the minimum of $(r-1)$ independent $[0,1]$-Uniforms. Consequently a.s. $\liminf_{t\to\infty}p(t)>0$.
\item[\textup{(ii)}] For $m>1$, almost surely $\lim_{t\to\infty} p(t)=1$.
\end{enumerate}
\end{customthm}

\begin{proof} By the definition of the UAM process, we have
\begin{equation}\label{a}
\Bbb P(X(t+1)=X(t)+1|\circ)= 1 - (1-p(t))^m.
\end{equation}
{\bf (i)\/} Consider $m=1$. For $r=1$, we have $p(t)\equiv 1$. Consider $r\ge 2$. The equation \eqref{a} is the case $\delta=-1$, $\ga=0$ of \eqref{jumpup}. By Lemma
\ref{Lem1}, we claim that $M(t):=\frac{(X(t))^{(\ell)}}{t^{(\ell)}}$ is a martingale. So arguing as in the proof of  Theorem 
\ref{th1}, we obtain that almost surely (a.s.) $\lim p(t)=p(\infty)$ exists, and the limiting distribution of $p(\infty)$
is a beta-distribution with
parameters $a=1$ and $b=r-1$. That is, the limiting density is $(r-1)(1-x)^{r-2}$, $x\in [0,1]$. Therefore a.s. $\lim p(t)>0$.

{\bf (ii)\/} Consider $m>1$. Clearly $G_{t,1}\subset G_{t,m}$. Therefore a.s. $\liminf p(t)>0$ as well. Furthermore, it follows from \eqref{a} that
\[
\Bbb E[p(t+1)|\circ]= p(t)+\frac{1}{t+1}\bigl[1-p(t)-(1-p(t))^m\bigr],
\]
which is a special case of \eqref{p(t+1)}, with $O(t^{-2})$ dropped. So we obtain that 
$\Bbb P\bigl(p(\infty)\in \{0,1\}\bigr)=1$, which in combination with $\Bbb P(p(\infty)>0)=1$ imply that
$\Bbb P(p(\infty)=1)=1$.
\end{proof}

\section{A technical lemma}

In this section, we will prove Lemma~\ref{general}. We need the following Chernoff bound for its proof. (See e.g.\ \cite[Theorem 2.8]{JLR}.)
\begin{theorem}\label{thm:Chernoff}
If $X_1,\dots,X_n$ are independent Bernoulli random variables, $X=\sum_{i=1}^n X_i$, and $\la = \mathbb{E}[X]$, then
\[
\mathbb P(|X-\la| > \eps\la) < 2\exp\brac{-\eps^2\la/3} \quad \forall \eps\in (0, 3/2).
\]
\end{theorem}

\begin{lemma}\label{general}
Let $\{X(t)\}_{t\ge 0}$ be a sequence of random variables such that $X(0)=0$ and $X(t+1)-X(t) \in \{0,1\}$. Let $x(t)=X(t)/t$ and assume
\begin{equation}\label{rec;ineq}
\Bbb E[x(t+1) - x(t)|\circ]\le \frac{h(x(t))}{t} +O(t^{-2}),
\end{equation}
where $h$ is a continuous, strictly decreasing function with $h(0)>0$ and $h(1)<0$, so that $h(x)$ has a unique root  $\rho \in(0,1)$. Assume also that $h'(x) < -1$ in $(0,1)$. Then, for any $\ga<1/3$, almost surely
\[
\lim_{t\to\infty} t^{\ga}\max\{0,x(t)-\rho\}=0.
\]
\end{lemma}

\begin{lemma}[Extensions of Lemma~\ref{general}]\label{extensions} 
Lemma~\ref{general} can be extended in a couple of ways as follows. 
\begin{enumerate}
\item[\textup{(a)}] If the hypothesis $X(t+1)-X(t) \in \{0,1\}$ in Lemma~\ref{general} is replaced with $X(t+1)-X(t) \in \{-1,1\}$, then the conclusion of Lemma~\ref{general} still holds. This follows from minor modifications in the proof of Lemma~\ref{general}.
\item[\textup{(b)}] If the inequality sign in \ref{rec;ineq} is replaced with an equality sign, .i.e. under the condition $\mean{x(t+1) - x(t) \mid \circ} = \frac{h(x(t))}{t} +O(t^{-2})$, we have the following conclusion:
for any $\ga<1/3$, almost surely
\[
\lim_{t\to\infty} t^{\ga} (x(t)-\rho)=0. 
\]
\begin{proof}[Proof of \textup{(b)}]
First of all, by Lemma~\ref{general}, we have $\lim_{t\to\infty} t^{\ga}\max\{0,x(t)-\rho\}=0$ almost surely.
Second, let $g(z) = -h(1-z)$, so that $g(0) >0$ and $g(1)<0$, and  in $(0,1)$, we have $g'(z) = h'(1-z)<-1$. Letting $y(t) = 1-x(t)$,
\begin{align*}
\mean{y(t+1) -y (t) \mid \circ} &= \mean{(1-x(t+1)) -(1- x(t)) \mid \circ} \\
&= -\frac{h(x(t))}{t} +O(t^{-2})\\
&= \frac{g(y(t))}{t} +O(t^{-2}).
\end{align*}
Applying Lemma~\ref{general} with $X_1(t) = t-X(t)$, and then switching back to $X(t)$, we see 
that $\lim_{t\to\infty} t^{\ga}\max\{0, \rho-x(t)\}=0$ almost surely, as well.
\end{proof}
\end{enumerate}
\end{lemma}

\medskip
\begin{proof}[Proof of Lemma~\ref{general}]
Let $\eps=\eps_t:=t^{-1/3}\log t$.  We will show
\begin{equation}\label{x_t<rho+eps}
\mathbb P(x(t)>\rho+\eps) \le \exp\brac{-\Theta\brac{\log^3t}}.
\end{equation}
Once we show~\eqref{x_t<rho+eps}, the Borel-Cantelli lemma gives 
\[
\mathbb P(x(t)-\rho >  t^{-1/3}\log t \quad \text{infinitely often})=0,
\]
which proves what we want. Let us prove~\eqref{x_t<rho+eps}. 

For $T\in [0,t)$, let $\mathcal E_T$ be the event that \{$x(t)> \rho+\eps$ and $T$ is the last time such that $X(\tau) \le (\rho+\eps/2)\tau$\}, that is, 
\[
X(T)\le (\rho+\eps/2)T; \quad x(\tau)> \rho+\eps/2\tau, \,\, \forall\, \tau\in (T,t); \quad x(t) >\rho+\eps.
\]
Since $X(t+1)-X(t) \in \{0,1\}$, we have
\begin{align*}
X(T)+t-T\ge X(t)>t(\rho+\eps).
\end{align*}
Using $X(t) = tx(t)$ above , we get
\[
T(\rho+\eps/2)+t-T>t(\rho+\eps),
\]
implying
\begin{equation}\label{ubT}
t-T > \frac{t\eps}{2(1-\rho)}.
\end{equation}
We conclude that 
\[
\{x(t)>\rho+\eps\}\subseteq \bigcup_{T=1}^s \mathcal E_T,\quad s=s(t):= t-\Big\lceil\frac{t\eps}{2(1-\rho)}\Big\rceil.
\]

Now let us fix a $T\in [0,s]$ and bound  $\mathbb  P(\mathcal E_T)$. 
The main idea of the proof is that, as long as $x(\tau)>\rho$, by Equation~\eqref{rec;ineq}, the process $\{x(\tau)\}$ has a negative drift.

Let $\xi_\tau$ denote the indicator of the event $\{x(\tau-1)>\rho+\eps/2 \text{ and } X(\tau)=X(\tau-1)+1\}$ and let  $\mathcal Z_T:=\xi_{T+2}+\cdots+\xi_t$.
On the event $\mathcal E_T$, the sum $\mathcal Z_T$ counts the total number of upward unit jumps ($X(\tau)-X(\tau-1)=1$,
$\tau\in [T+2,t]$) and therefore
\[
X(T+1) + \mathcal Z_T = X(t)\ge t(\rho+\eps).
\]
Since $X(T+1) \le X(T)+1\le T(\rho+\eps/2)+1$, we must have
\begin{equation*}
Z_T> (\rho+\eps)(t-T),\quad Z_T:=1+\mathcal Z_T.
\end{equation*}
Writing $p(x(\tau)):= \pr{X(\tau+1) = X(\tau)+1 \mid \circ}$,
\begin{align*}
\mean{x(\tau+1) \mid \circ} &= p(x(\tau)) \frac{X(\tau)+1}{\tau+1} + (1-p(x(\tau))) \frac{X(\tau)}{\tau+1} \\
&=  \frac{p(x(\tau))}{\tau+1} + \frac{\tau x(\tau)}{\tau+1} \\
& = x(\tau) + \frac{p(x(\tau)) - x(\tau)}{\tau+1},
\end{align*}
so that $p(x(\tau)) = h(x(\tau)) + x(\tau) + O(\tau^{-1})$.

Recall that, for $\tau \ge T+1$, we have $x(\tau) > \rho +\eps/2$. Since $h'(x)<-1$ in $(0,1)$, the sum $h(x) + x$ is decreasing in $(0,1)$. Hence, conditioning on the full record (up to and including time $\tau$),
\begin{align*}
\mathbb P(\xi_{\tau+1}=1|\circ) &=\mathbb P(X(\tau+1)=X(\tau)+1\, |\,\circ)\\
&= h(x(\tau)) + x(\tau) +O\brac{\tau^{-1}} \\
&< h(\rho +\eps/2) + \rho + \eps/2 + O\brac{\tau^{-1}}\\
&= h(\rho) + (\eps/2)\cdot h'(y) +\rho +\eps/2 + O\brac{\tau^{-1}} \quad \text{for some }y\in (\rho,\rho+\eps/2)\\
& < \rho + O\brac{\tau^{-1}}.
\end{align*}
Hence, the sequence $\{\xi_\tau\}$ is stochastically dominated by the sequence of {\it independent\/} Bernoulli random variables $B_\tau$ with parameters $\min\bigl(\rho+O(\tau^{-1}),1\bigr)$.
Consequently, $Z_T$ is stochastically dominated by $1+\sum_{j=T+2}^t B_j$, and
\[
\la:=\sum_{j=T+2}^t \mathbb E[B_j]=\rho(t-T) +O(\log t).
\]
For the choice of $\eps$ we have, \eqref{ubT}  gives
\[
(\rho+\eps)(t-T) \ge (1+\eps/2)\lambda.
\]
Thus, by the Chernoff bound in Theorem~\ref{thm:Chernoff} and using~\eqref{ubT},
\begin{align*}
\pr{\mathcal{E}_T} 
&\le \mathbb P(Z_T>(t-T)(\rho+\eps)) \\
&\le \mathbb P\Big(1+B_{T+2}+\cdots+B_t>(t-T)(\rho+\eps)\Big) \\
&\le \mathbb P\brac{1+B_{T+2}+\cdots+B_t> (1+\eps/2)\lambda} \\
&\le \exp\brac{-\Theta(\eps^2(t-T))}\le  e^{-\Theta(\log^3 t)}.
\end{align*}
Using the union bound on $T$ we complete the proof of~\eqref{x_t<rho+eps} and of the lemma.
\end{proof}

\section{Greedy Matching Algorithm} 
Recall that the greedy matching algorithm (for either of two graph models) generates the increasing sequence $\{ M(t)\}$ of partial matchings on the sets $[t]$, with $ M(1)=\emptyset$. Given $M(t)$, let
\begin{align*}
X(t)&:= \text{number of unmatched vertices at time }t,\\
Y(t)&:= \text{total degree of unmatched vertices at time }t,\\
U(t)&:=\text{number of unmatched vertices selected by $t+1$  from $[t]\setminus M(t)$},\\
x(t)&:=X(t)/t,\\
y(t)&:=Y(t)/(2mt).
\end{align*}

\subsection{The PAM case} 

\begin{customthm}{2.5}
Let $X(t)$ be the number of unmatched vertices at time $t$ in the greedy matching algorithm.
For $\delta>-m$, let $\rho_{m,\delta}$ be  the unique root in $(0,1)$ of
\begin{equation}\label{defh}
h(z)=h_{m,\de}(z):= 2\left[1-\bfrac{m+\de}{2m+\de}z\right]^m-z-1.
\end{equation}
Then, for any $\a<1/3$, almost surely,
\beq{th3repeatconc}
{
\lim_{t\to\infty} t^{\a}\max\{0, x(t)-\rho_{m,\delta}\}=0.
}
In consequence, the Greedy Matching Algorithm a.s. finds a sequence of nested matchings $\{M(t)\}$, 
where the number of vertices in $M(t)$ is asymptotically at least $(1-\rho_{m,\delta})t$.
\end{customthm}

\begin{proof}
Notice first that, for $\delta>-m$, the function $h(z)$ is decreasing on $(0,1)$ and $h(z)=0$ does have a unique solution in the same interval. 

We will prove our claim first for a slightly different model that does not allow any loops other than at the first vertex. In this model, vertex $1$ has $m$ loops, and the $i$-th edge of vertex $t+1$ attaches to $u\in [t]$ with probability
\[
\frac{d_{t,i-1}(u)+\delta}{2mt+2(i-1)+t\delta}.
\]

\medskip
\noindent \textbf{Loops not allowed except at vertex 1.}
In this case, since each degree is at least $m$, we have $Y(t)\ge mX(t)$ and hence $y(t)\ge x(t)/2$.
Also, since
\[
X(t+1)=
\begin{cases}
X(t)+1 & \text{if } U(t)=0\\
X(t)-1 & \text{if } U(t)>0,
\end{cases}
\]
we have
\begin{equation}\label{meanX}
\mathbb E[X(t+1)|\circ]= X(t)+ \mathbb P(U(t)=0|\circ)-\mathbb P(U(t)>0|\circ).
\end{equation}
Since $\mathbb P(\text{vertex $t+1$ has some loop})= O(t^{-1})$, using $Y(t)\ge mX(t)$ in the last step below, 
by Lemma \ref{lem1} we get
\begin{align}\label{P(U=0)}
\mathbb P(U(t)=0|\circ)
&= \mathbb P(U(t)=0 \text{ and vertex $t+1$ has no loop} |\circ) +O\brac{t^{-1}}	\notag\\
&= \brac{1-O\brac{t^{-1}}}\frac{(2mt-Y(t)+\de t-\de X(t))^{(m)}}{(2mt+\de t)^{(m)}}+O\brac{t^{-1}}	\notag\\
&=\frac{(2mt+\de t-Y(t)-\de X(t))^{m}}{(2mt+\de t)^{m}}+O\brac{t^{-1}}	\notag\\
&=\brac{1-\frac{2m}{2m+\de}\,y(t)-\frac{\de}{2m+\de}\,x(t)}^m+O\brac{t^{-1}}	 \\
&\le \left(1-\frac{m+\de}{2m+\de}\,x(t)\right)^m+ O(t^{-1}). \notag  
\end{align}
Using \eqref{meanX} and \eqref{P(U=0)} gives
\begin{align}\label{meanXupper}
\mathbb E[x(t+1)|\circ] &\le x(t)+ \frac1t\left[2\brac{1-\frac{m+\de}{2m+\de}\,x(t)}^m-x(t)-1\right] +O\brac{t^{-2}} \notag \\
&=x(t)+\frac{1}{t}h(x(t))+O\brac{t^{-2}},
\end{align}
where $h(z)$ is as defined in~\eqref{defh}.
Note that $X(0) = 0$ and $h(z)$ satisfies the conditions given in Lemma~\ref{general} and Lemma~\ref{extensions}, namely, $h(0)>0$, $h(1)<0$, and $h'(z) <- 1$ for $z\in (0,1)$.
The conclusion of the theorem follows from the first part of Lemma~\ref{extensions} in this case.

\medskip
\noindent \textbf{Loops allowed everywhere.} The above analysis is carried over to this more complicated case via an argument  similar to
the one for the descendant trees in the subsections 1.1.  Here is a proof sketch. First, the counterpart of \eqref{P(U=0)} is: 
\begin{equation*}
\begin{aligned}
&\mathbb P(\!\{U(t)=0\}\!\cap\! \{\text{no loops at }t+1\}|\circ)\\
&=\Pi_m(t)\prod_{j=0}^{m-1}\bfrac{2mt-Y(t)+\de t-\de X(t)+j}{2mt+2j+1+\de t+(j+1)\,\de/m}\\
&\le \Pi_m(t)\left[\left(1-\frac{m+\de}{2m+\de}\,x(t)\right)^m+ O(t^{-1})\right]\\
&=\bigl(1-O(t^{-1})\bigr)\left[\left(1-\frac{m+\de}{2m+\de}\,x(t)\right)^m+ O(t^{-1})\right]\\
&= \left(1-\frac{m+\de}{2m+\de}\,x(t)\right)^m+O(t^{-1});
\end{aligned}
\end{equation*}
see \eqref{Pmt=} for $\Pi_m(t)$. Therefore we obtain again \eqref{meanXupper}. The rest of the proof remains the same. 
\end{proof}

\begin{remark}
Let $r=r_{m,\de}:=1-\rho_{m,\de}$, where $\rho_{m,\de}$ is the unique root in $(0,1)$ of
\[
h(z)=h_{m,\de}(z):= 2\left[1-\bfrac{m+\de}{2m+\de}z\right]^m-z-1.
\]
Then, $r$ is the unique root in $(0,1)$ of
\[
f(z)=f_{m,\de}(z):= 2-z-2\brac{\frac{m}{2m+\de}+\frac{m+\de}{2m+\de}\,z}^m.
\]
Thus, by Theorem~\ref{th3}, we have
\[
\liminf (1-x(t)) \ge r
\]
almost surely, where $1-x(t)$ is the fraction of the vertices in $M(t)$. See~\eqref{rvalues} for various $r$ values when $\delta=0$.
\end{remark}

\begin{remark}
When $\delta \to \infty$, the function $f_{m,\delta}(z)$ as defined above converges to $2-z-2z^m$ in $(0,1)$. So it is plausible that for the case of uniform attachment model, the number of vertices in $M(t)$ is asymptotically $rt$, where $r$ is the unique root of $2-z-2z^m$. This is in fact the case as shown in the next theorem.
\end{remark}

\subsection{The UAM case}

\begin{customthm}{2.8}
Let $M(t)$ denote the greedy matching set after $t$ steps of the UAM process. 
Let $r_m$ denote a unique positive root of $2(1-z^m) -z=0$: $r_m=1-m^{-1}\log 2+O(m^{-2})$. 
Then, for any $\a<1/3$, almost surely
\[
\lim_{t\to\infty} t^{\a}\biggl| \frac{2|M(t)|}{t}-r_m\biggr|=0.
\]
\end{customthm}

\begin{proof}
Let $X(t) = t-2|M(t)|$ as before. In particular, we have $X(0)=0$ and $X(1)=1$.
At each step $t\ge 2$, we check the edges incident to vertex $t$.
If some of the edges end at vertices that do not belong to $M(t)$, then we choose the largest (youngest) of those vertices, say $w$, and set $M(t):=M(t-1)\cup \{(t,w)\}$ and $X(t)=X(t-1)-1$.  Otherwise, $M(t)=M(t-1)$ and $X(t)=X(t-1)+1$. Let $x(t)=X(t)/t$ be the fraction of unmatched vertices after step $t$. Then $X(t)$ is a Markov chain with
\[
\pr{X(t+1)-X(t) =1 \mid X(t)} = (1-x(t))^m
\]
since for $X(t+1)-X(t) =1$ to happen, each of the $m$ choices made by vertex $t+1$ must lie outside of $M(t)$, the probability of each such choice is $1-x(t)$, and the choices are independent of each other.
%
With the remaining probability vertex $t+1$ chooses at least one of $m$ vertices from $X(t)$, in which case $X(t)$ decreases by 1. Consequently,
\[
\mean{X(t+1)\big|\circ} = X(t) + (1-x(t))^m - (1-(1-x(t))^m) = X(t) +2(1-x(t))^m-1.
\]
Dividing both sides with $t+1$, we obtain
\[
\mean{x(t+1) | \circ} = x(t) + \frac{h(x(t))}{t+1} - O(1/t^2),
\]
where $h(z) = 2(1-z)^m-z-1$.
The function $h$ meets the conditions of the second part of Lemma~\ref{extensions}. 
Hence $\lim_{t\to\infty} t^{\ga}|x(t)-\rho_m| = 0$ a.s.
On the other hand, if $\rho_m$ is the unique root of $h$ in $(0,1)$, then $r_m:=1-\rho_m$ is the unique root of $2(1-z^m)-z$ in $(0,1)$.
Since $x(t)-\rho_m = (1 -2 |M(t)|/t) - (1-\rho_m)  = \rho_m -2 |M(t)|/t$, we also have, a.s.
\[
\lim_{t\to\infty} t^{\ga}\Big|2|M(t)|/t-r_m\Big| = 0.
\]

This  completes the proof.
\end{proof}

\section{Analysis of Greedy Independent Set Algorithm} The algorithm, for both PAM and UAM cases, generates the increasing
sequence  of independent sets $\{I(t)\}$ on the sets $[t]$, with $I(1):=\{1\}$. If vertex $t+1$ does not select a single vertex from $t$, we set $I(t+1)=I(t)\cup \{t+1\}$; otherwise $I(t+1):= I(t)$. Given $I(t)$, let 
\begin{align*}
X(t)&:= \text{number of vertices }\le t \text{ outside of the current independent set } I(t),\\
Y(t)&:=\text{total degree of these outsiders} ,\\
Z(t)&:=\text{number of insiders selected by outsiders by time }t,\\
U(t)&:=\text{number of insiders selected by vertex } t+1,\\
x(t) &:=\frac{X(t)}{t},\quad y(t):=\frac{Y(t)}{2mt},\quad z(t)=\frac{Z(t)}{mt}, \quad i(t)=\frac{|I(t)|}{t}.
\end{align*}
Since each insider selects only among outsiders, the total degree of insiders  is $m|I(t)| +Z(t)$ and
\[
Y(t)=2mt- m|I(t)|-Z(t)\Longrightarrow y(t)=1- (i(t)+z(t))/2.
\]

\subsection{The PAM case} 

\begin{customthm}{2.9} Let $w_m$ denote the unique root of $-w+(1-w)^m$ in $(0,1)$. For any 
$\chi\in \Bigl(0,\min\Bigl\{\frac{1}{3},\frac{2m+2\delta}{3(2m+\delta)}\Bigr\}\Bigr)$, almost surely
\begin{equation}\label{PAM0}
\lim_{t\to\infty}t^{\chi}\biggl|\frac{|I(t)|}{t}-w_m\biggr|=0.
\end{equation}
\end{customthm}

\begin{proof} By the definition of the algorithm, we have
\[
|I(t+1)|=\left\{\begin{aligned}
&|I(t)|+1,&&\text{with probability }\Bbb P(U(t)=0|\circ),\\
&|I(t)|,&&\text{with probability }\Bbb P(U(t)>0|\circ).\end{aligned}\right.
\]
Now $U(t)=0$ means that vertex $t+1$ selects all $m$ vertices from the outsiders set, so that
\begin{align*}
\Bbb P(U(t)=0|\circ)&=\frac{\bigl(Y(t)+\delta X(t)\bigr)^{(m)}}{\bigl((2m+\delta)t\bigr)^{(m)}}+O(t^{-1})
=\biggl(\frac{2my(t)+\delta x(t)}{2m+\delta}\biggr)^m +O(t^{-1})\\
&=\biggl(1-\frac{m+\delta}{2m+\delta}\, i(t)-\frac{m}{2m+\delta}\,z(t)\biggr)^m +O(t^{-1}).
\end{align*}
The leading term in the first equality is the exact (conditional) probability of $\{U(t)=0\}$ when no loops at vertices other than the first vertex are allowed, and the extra $O(t^{-1})$ is for our, more general, case when the loops are admissible. 
So using $i(t)=|I(t)|/t$, we have 
\begin{multline}\label{PAM6}
\Bbb E[i(t+1)|\circ]=i(t)
+\frac{1}{t+1}\biggl[-i(t)
+\biggl(1-\frac{m+\delta}{2m+\delta} i(t)-\frac{m}{2m+\delta}z(t)\biggr)^m\biggr]+O(t^{-2}).
\end{multline}
Next $Z(t+1)=Z(t)+U(t)$, so that
\begin{align*}
\Bbb E[Z(t+1)|\circ]&= Z(t) +\Bbb E[U(t)|\circ]\\
&=Z(t)+m\frac{I(t)(m+\delta)+Z(t)}{(2m+\delta)t}+O(t^{-1}),
\end{align*}
and using $z(t)=Z(t)/mt$, we have
\begin{align}\label{PAM7}
\Bbb E[z(t+1)|\circ]=z(t) +\frac{1}{t+1}\biggl[-z(t) + \biggl(i(t)\frac{m+\delta}{2m+\delta}+z(t)\frac{m}{2m+\delta}\biggr)
\biggr] +O(t^{-2}).
\end{align}
Introduce $w(t)=i(t)\frac{m+\delta}{2m+\delta}+z(t)\frac{m}{2m+\delta}$.  Multiplying  the equations \eqref{PAM6}
and \eqref{PAM7} by $\frac{m+\delta}{2m+\delta}$ and by $\frac{m}{2m+\delta}$ and adding the products, we obtain 
\begin{equation}\label{PAM8}
\begin{aligned}
\Bbb E[w(t+1)|\circ]&=w(t)+\frac{f(w(t))}{t+1}+O(t^{-2}),\\
f(w)&:=\frac{m+\delta}{2m+\delta}\bigl[-w+(1-w)^m\bigr].
\end{aligned}
\end{equation}
The function $f$ is qualitatively similar to the function $h$ in the proof of Theorem \ref{th3repeat}.
Indeed $f(w)$ is strictly decreasing, with $f(0)=1$ and $f(1)=-1$.
Therefore $f(w)$ has a unique root $w_m\in (0,1)$; it is not difficult to see that
\[
w_m=\frac{\log m}{m}\Bigl[1+O(\log\log m/\log m)\Bigr],\quad m\to\infty.
\]
Let us prove that, for any $\a< c(m,\delta):=\min\Bigl\{1,\frac{2m+2\delta}{2m+\delta}\Bigr\}(\subseteq (0,1])$, and $A\ge A(\a)$, we have
\begin{equation}\label{PAM8.1}
\Bbb E\bigl[(w(t)-w_m)^2\bigr]\le A t^{-\a},
\end{equation}
First
\begin{equation}\label{PAM8.15}
\begin{aligned}
|i(t+1)-i(t)|&=\left|\frac{|I(t+1)|}{t+1}-\frac{|I(t)|}{t}\right|\\
&\le \frac{1}{t+1}\bigl(|I(t+1)|-|I(t)|\bigr) +\frac{|I(t+1)|}{t+1}\le \frac{2}{t+1},
\end{aligned}
\end{equation}
and similarily
\begin{equation*}
|z(t+1)-z(t)|\le \frac{2}{t+1}.
\end{equation*}
Therefore $|w(t+1)-w(t)|\le 2/(t+1)$, and consequently
\[
\bigl(w(t+1)-w_m\bigr)^2\le \bigl(w(t)-w_m\bigr)^2+\frac{4}{t^2}+2(w(t)-w_m) \bigl(w(t+1)-w(t)\bigr).
\]
So, conditioning on prehistory, we have 
\begin{multline*}
\Bbb E\bigl[(w(t+1)-w_m)^2|\circ]\le \bigl(w(t)-w_m\bigr)^2\\
+2\bigl(w(t)-w_m\bigr) \Bbb E\bigl[w(t+1)-w(t)|\circ] +O(t^{-2})\\
=\bigl(w(t)-w_m\bigr)^2 +\frac{2\bigl(w(t)-w_m\bigr)}{t+1} f(w(t))+O(t^{-2}).
\end{multline*}
By \eqref{PAM8}, $f'(w)\le -\frac{m+\delta}{2m+\delta}$; therefore, with $c(m,\delta)=\frac{2m+2\delta}{2m+\delta}$, the last inequality gives
\[
\Bbb E\bigl[(w(t+1)-w_m)^2|\circ]\le \left(\!1-\frac{c(m,\delta)}{t+1}\right)\!\bigl(w(t)-w_m\bigr)^2 +O(t^{-2}),
\]
leading to a recursive inequality 
\begin{equation}\label{PAM8.2}
\Bbb E\bigl[(w(t+1)-w_m)^2]\le \left(\!1-\frac{c(m,\delta)}{t+1}\right)\!\Bbb E\bigl[(w(t)-w_m)^2\bigr] +O(t^{-2}).
\end{equation}
The bound \eqref{PAM8.1} follows from \eqref{PAM8.2} by a straightforward induction on $t$.
Next we use \eqref{PAM8.1} to  prove that, for $A_1$ large enough,
\begin{equation}\label{PAM8.21}
\Bbb E\bigl[(i(t+1)-w_m)^2|\circ]\le  A_1 t^{-\a}.
\end{equation}
Using \eqref{PAM8.15}, we have
\[
\bigl(i(t+1)-w_m\bigr)^2\le \bigl(i(t)-w_m\bigr)^2+\frac{4}{t^2}+2(i(t)-w_m) \bigl(i(t+1)-i(t)\bigr).
\]
Consequently
\begin{multline*}
\Bbb E\bigl[(i(t+1)-w_m)^2|\circ]\le \bigl(i(t)-w_m\bigr)^2\\
+2\bigl(i(t)-w_m\bigr) \Bbb E\bigl[i(t+1)-i(t)|\circ] +O(t^{-2})\\
=\bigl(i(t)-w_m\bigr)^2 +\frac{2\bigl(i(t)-w_m\bigr)}{t+1} \bigl[-i(t)+(1-w(t))^m\bigr]+O(t^{-2}).
\end{multline*}
Taking expectations of both sides, and using Cauchy's inequality and \eqref{PAM8.1}, we obtain
\begin{multline*}
\Bbb E\bigl[(i(t+1)-w_m)^2]\le \biggl(\!1-\frac{2}{t+1}\biggr)\Bbb E[(i(t)-w_m)^2]\\
+\frac{2}{t+1}\Bbb E^{1/2}\bigl[(i(t)-w_m)^2]\,\Bbb E^{1/2}\bigl[(w_m-(1-w(t))^m)^2\bigr]+O(t^{-2})\\
\le \biggl(\!1-\frac{2}{t+1}\biggr)\Bbb E[(i(t)-w_m)^2] +\frac{2mA^{1/2}}{t^{1+\a/2}}\Bbb E^{1/2}\bigl[(i(t)-w_m)^2\bigr]
+O(t^{-2}),
\end{multline*}
since $w_m=(1-w_m)^m$, and 
\[
\big|(1-w(t))^m - (1-w_m)^m\big|\le m|w(t)-w_m|.
\]
So it suffices to show existence of $A_1$ such that 
\[
\biggl(\!1-\frac{2}{t+1}\biggr) A_1 t^{-\a}+\frac{2mA^{1/2} A_1^{1/2}t^{-\a/2}}{t^{1+\a/2}}+O(t^{-2})\le A_1 (t+1)^{-\a}.
\]
holds for $t>t_0$ where $t_0$ depends only on $A$ and $\a$. For large $t$, the above inequality becomes
\[
2mA^{1/2}A_1^{1/2}+O(t^{-1-\a})\le A_1\bigl[(2-\a)+O(t^{-1})\bigr]+O(t^{-1+\a}),
\]
and $A_1> A\Bigl(\frac{2m}{2-\a}\Bigr)^2$ does the job. So \eqref{PAM8.21} is proved.
Therefore, by Markov inequality,
\begin{equation}\label{0.99}
\Bbb P\bigl(|i(t)-w_m|\ge t^{-\chi}\bigr)\le At^{-\a+2\chi}\to 0,\quad\chi\in (0,\a/2).
\end{equation}
This inequality already means that $i(t)\to  w_m$ {\it in probability\/}. Let us show that, considerably stronger, $i(t)\to w_m$ with probability $1$, at least as fast as $t^{-\chi}$, for any given $\chi<1/3$.
Pick $\be>1$ and introduce a sequence $\{t_{\nu}\}$, $t_{\nu}=\lfloor \nu^{\be}\rfloor$. By \eqref{0.99}, we have
\begin{align*}
\sum_{\nu\ge 1}\Bbb P\bigl(|i(t_{\nu})-w_m|\ge t_{\nu}^{-\chi}\bigr)&\le \sum_{\nu\ge 1}A_1t_{\nu}^{-\a+2\chi}
=O\biggl(\sum_{\nu\ge 1}\nu^{-\be (\a-2\chi)}\biggr)<\infty,
\end{align*}
provided that $\be>(\a-2\chi)^{-1}$, which we assume from now. For such choice of $\be$, by Borel-Cantelli lemma
with probability $1$ for all but finitely many $\nu$ we have $|i(t_{\nu})-\rho|\le t_{\nu}^{-\chi}$. Let 
$t\in [t_{\nu}, t_{\nu+1}]$. By \eqref{PAM8.15}, we have
\[
|i(t)-i(t_{\nu}))|=O\Bigl(\frac{t_{\nu+1}-t_{\nu}}{t_{\nu}}\Bigr),
\]
uniformly for all $\nu$.
So if $|i(t_{\nu})-\rho|\le t_{\nu}^{-\chi}$, then (using $t_{\nu}=\Theta(\nu^{\be})$), we have: for $t\in [t_{\nu},t_{\nu+1}]$,
\begin{align*}
|i(t)-w_m|&\le t_{\nu}^{-\chi}+O\biggl(\frac{t_{\nu+1}-t_{\nu}}{t_{\nu}}\biggr)\\
&=O\bigl(\nu^{-\be\chi}+\nu^{-1}\bigr)=O\bigl(\nu^{-\min(\be\chi,1)}\bigr)\\
&=O\bigl(t^{-\min(\chi,\be^{-1})}\bigr).
\end{align*}
Since $|i(t_{\nu})-w_m|\le t_{\nu}^{-\chi}$ holds almost surely (a.s.) for all but finitely many $\nu$'s, we see then
that a.s. so does the bound $|i(t)-w_m|=O\bigl(t^{-\min(\chi,\be^{-1})}\bigr)$ for all but finitely
many $t$'s. Now by taking $\be$ sufficiently close to $(\a-2\chi)^{-1}$ from above, we can make $\min(\chi,\be^{-1})$
arbitrarily close to $\min(\chi, \a-2\chi)$ from below. It remains to notice that $\min(\chi, \a-2\chi)$ attains its
maximum $\a/3$ at $\chi=\a/3$. The proof of Theorem \ref{thm5} is complete.
\end{proof}

\subsection{The UAM case}

\begin{customthm}{2.11} Let $w_m$ be the unique root of $-w+(1-w)^m$ in $(0,1)$. Then, for any $\a<1/3$, almost surely
\[
\lim_{t\to\infty}t^{\a}\Big|\frac{|I(t)|}{t} - w_m\Big|=0.
\]
\end{customthm}

This following remark is already stated as Remark~\ref{w_m}.
\begin{remark} 
Thus, the convergence rate aside, almost surely the greedy independent algorithm delivers a sequence of independent sets of asymptotically the same size as for the PAM case.
\end{remark}

\begin{proof} Let $i(t)=|I(t)|/t$. From the definition of the UAM process and the greedy independent set algorithm, we obtain
\[
|I(t+1)|=\left\{\begin{aligned}
&|I(t)|+1,&&\text{with conditional probability } (1-i(t))^m,\\
&|I(t)|,&&\text{with conditional probability } 1-(1-i(t))^m.\end{aligned}\right.
\]
Therefore
\[
\Bbb E\Bigl[|I(t+1)|\big|\circ\Big]=|I(t)| +(1-i(t))^m,
\]
or equivalently
\[
\Bbb E[i(t+1)|\circ]=i(t)+\frac{1}{t+1}\bigl[-i(t)+(1-i(t))^m\bigr].
\]
The function $-x+(1-x)^m$ differs by a constant positive factor from the function $f$ in \eqref{PAM8}. 
The function $f$ meets the conditions of Lemma \ref{general}, and $r_m$ is a unique root of $f$.
Therefore, for any $\a<1/3$, a.s. $\lim_{t\to\infty} t^{\a}|i(t)-r_m|=0$.
\end{proof}

\section*{Appendix}

In order to show that the coupling described in Section~\ref{sec:descendant} really works, we can compute the probability that the $i$-th edge of vertex $w_{t+1}$ connects to vertex $w_x$ in the coupling and compare it with the probability in~\eqref{defprob}. Let us denote by $\{G'_{m,\de}(t)\}$ the process obtained by collapsing the vertices of $\{G_{1,\de/m}(mt)\}$.  Note that $(mt+i)$-th edge of the $\{G_{1,\de/m}(mt)\}$-process becomes the $i$-th edge of $w_{t+1}$ after the collapsing.
Hence the $i$-th edge of vertex $w_{t+1}$ connects to $w_x$ ($x\le t$) if and only if the $(mt+i)$-th edge of $\{G_{1,\delta/m}(mt)\}$-process connects $v_{mt+i}$ with one of the vertices $v_{m(x-1)+1},\dots,v_{mx}$. 
Let us denote by $d_{mt+i-1}(v_y)$ the degree of $v_y$ ($y\le mt+i$) just before the $(mt+i)$-th edge of $\{G_{1,\de/m}\}$-process is drawn.
Also, let $D_{t,i-1}(w_x)$ denote the degree of $w_x$ at the exact same time.
Hence, by definition,
\[
D_{t,i-1}(w_x)=
\begin{cases}
\displaystyle \sum_{y=mx-m+1}^{mx} \ d_{mt+i-1}(v_y), & x\le t\\
\displaystyle \sum_{y=mt+1}^{mt+i}\  d_{mt+i-1}(v_y), & x=t+1.
\end{cases}
\]
By~\eqref{n1}, for $x\le t$, the probability that $v_{mt+i}$ connects to one of the vertices $v_{m(x-1)+1},\dots,v_{mx}$ (equivalently, the probability that the $i$-th edge of $w_{t+1}$ connects to $w_x$) is
\begin{align*}
\frac {\sum_{y=mx-m+1}^{mx} \ \brac{d_{mt+i-1}(v_y)+\de/m}} {(2+\delta/m)(mt+i-1)+1+\delta/m} &= \frac{\delta+ \sum_{y=mx-m+1}^{mx}\ d_{mt+i-1}(v_y)}{\delta(t+i/m)+2mt+2i-1}\\
&=\frac{\delta+ D_{t,i-1}(w_x)}{\delta(t+i/m)+2mt+2i-1}.
\end{align*}
Similarly, the probability that $v_{mt+i}$ selects one of $v_{mt+1},\dots,v_{mt+i}$ (equivalently, the probability that the $i$-th edge of $w_{t+1}$ is a loop) is
\[
 \frac {1+i\de/m+\sum_{j=1}^{i-1} d_{mt+i-1}(v_{mt+j})}  {\delta(t+i/m)+2mt+2i-1}= \frac {1+i\de/m+D_{t,i-1}(w_{t+1})}  {\delta(t+i/m)+2mt+2i-1}.
\]
Note that the two probabilities above are the same as those in~\eqref{defprob} if we replace $D_{t,i-1}(w_x)$ with $d_{t,i-1}(x)$. Moreover, the two processes, $\{G'_{m,\de}(t)\}$ and $\{G_{m,\de}(t)\}$ as defined by~\eqref{defprob}, both start with $m$ loops on the first vertex, which implies $d_{1,0}(\cdot)=D_{1,0}(\cdot)$. This gives us that $\{G_{m,\de}(t)\}$ and $\{G'_{m,\de}(t)\}$ are equivalent processes, that is, at every stage, they produce the same random graph.

\end{document}